\documentclass[12pt]{amsart}
\usepackage{amsmath, amsfonts}
\usepackage{enumerate}
\numberwithin{equation}{section}
\newtheorem{theorem}{Theorem}[section]

\newtheorem{lemma}[theorem]{Lemma}

\begin{document}

\author{Ajai Choudhry}
\title[A diophantine problem]{A Diophantine Problem concerning\\
 Third Order Matrices}
\date{}

\begin{abstract} 
In this paper we find a third order unimodular matrix, none of whose entries is $1$ or $-1$, such that when each entry of the matrix is replaced by its cube, the resulting matrix is also unimodular. Further, we  find third order square integer matrices $(a_{ij})$, none of the integers $a_{ij}$ being $1$ or $-1$,  such that $\det{(a_{ij})}=k$ and  $\det{(a_{ij}^3)}=k^3$, where $k$ is a nonzero integer.
\end{abstract}

\maketitle

\smallskip


\section{Introduction}\label{intro}
This paper is concerned with the problem of finding a $3 \times 3$ integer matrix $(a_{ij})$, with no $a_{ij} =\pm 1$, such that $\det{(a_{ij})}=1$, and further, when each entry of the matrix $A$ is replaced by its cube, then also the determinant is 1, that is, $\det{(a_{ij}^3)}=1$. We also consider the more general problem of finding  a matrix $(a_{ij})$, none of the $a_{ij}$ being 0 or $\pm 1$,  such that  $\det{(a_{ij})}=k$ and   $\det{(a_{ij}^3)}=k^3$, where $k$ is a nonzero  integer.

 It is pertinent to recall that Molnar \cite{Mo} had posed the problem of finding an  $n \times n$ integer matrix $(a_{ij})$, with no $a_{ij} =\pm 1$, such that   $\det{(a_{ij})}=1$ and also   $\det{(a_{ij}^2)}=1$. Several authors found solutions of the problem when $n=3$ \cite{Gu1, Gu2, Gu3}.   In fact,  D\u{a}nescu, V\^aj\^aitu, and Zaharescu \cite{DVZ} solved Molnar's problem for matrices of arbitrary  order.   Guy  restricted the problem to $3 \times 3$ matrices in his book,   ``Unsolved problems in number theory'' \cite[Problem F28, pp.\ 265--266]{Gu4},   but  imposed  the additional condition  that all the entries $a_{ij}$ should also be nonzero. He concluded his discussion  by asking, ``Will the problem extend to cubes?''.  This question has, until now, remained completely unanswered. 

If $A=(a_{ij})$ is any $n \times n$ matrix, we will write $A^{(3)}$ to denote the matrix $(a_{ij}^3)$ obtained by  replacing each entry of the  matrix $A$ by its cube\footnote{This notation is adapted from the notation used by  D\u{a}nescu et al. \cite{DVZ}.}.  
We obtain in this paper a $3 \times 3$  matrix $A=(a_{ij})$ whose entries are  univariate polynomials with integer coefficients, with no $a_{ij} =\pm 1$, and such that both   $\det{A}$ and   $\det{(A^{(3)})}$ are equal to 1. We also obtain a parametric solution of  the more general problem of finding  a $3 \times 3$ integer matrix $(a_{ij})$, none of the $a_{ij}$ being 0 or $\pm 1$, and such that    $\det{(a_{ij})}=k$ and     $\det{(a_{ij}^3)}=k^3$, where $k$ is a nonzero integer.

\section{Unimodular matrices that remain unimodular \\when each entry is replaced by its cube}\label{unimodmatrices}

If $M$ is any $n \times n$ matrix such that both $\det{M}=1$ and $\det{(M^{(3)})}=1$, then several other matrices satisfying these conditions can readily be derived from the matrix $M$. In Section \ref{genlemma} we give a lemma that lists out such matrices, and in Section \ref{33matrices} we obtain third order matrices satisfying such conditions.

\subsection{A general lemma}\label{genlemma}
We will denote the transpose of a matrix $M$ by $M^T$. Further, we will write $E_{ij}$ to denote the elementary matrix obtained by interchanging the $i$-th and $j$-th rows of the identity matrix, and $E_i(\alpha)$ to denote the elementary matrix obtained by multiplying the $i$-the row of the identity matrix by $\alpha$. 
\begin{lemma} If $M$ is any $n \times n $ integer matrix with the property that both $\det{M}=1$ and $\det({M^{(3)})}=1$, then the following integer  matrices, derived from the matrix $M$, also have this property:
\begin{enumerate}[(i)]
\item the matrix $M^{T}$;

\item the matrices $M_1=E_{i_1}(-1)E_{i_2}(-1)M$ and $M_2=ME_{i_1}(-1)E_{i_2}(-1)$ where $i_1, i_2 \in \{1, \ldots, n\}$ such that $i_1 \neq i_2$;

\item the matrices $M_3=E_{i_1i_2} E_{j_1j_2}M, M_4=ME_{i_1i_2} E_{j_1j_2}$ and \\
$M_5=E_{i_1i_2}ME_{j_1j_2}$ where $i_1 \neq i_2$, $j_1 \neq j_2$ and $i_1, i_2$, $j_1, j_2 \in \{1, \ldots, n\}$; 

\item the matrix $M([i,j],\alpha)=E_i(\alpha)ME_j(\alpha^{-1})$ where $i, j \in \{1, \ldots, n\}$ and $\alpha$ is a nonzero rational number so chosen that the entries of the matrix $M([i,j],\alpha)$ are all integers. 
\end{enumerate}
\end{lemma}
\begin{proof} Clearly, $\det{(M^T)}=1$, and $(M^T)^{(3)}= (M^{(3)})^T$, hence $\det{((M^T)^{(3)})}$ $=\det{((M^{(3)})^T)}=\det{(M^{(3)})}=1$, which proves the first part of the lemma. To prove {\it (ii)},  we note that $\det{(E_{i_1}(-1))}=\det{(E_{i_2}(-1))}=-1$, hence $\det{M_1}=\det{M}$, and  by the definition of $M_1^{(3)}$, it follows that $M_1^{(3)}=E_{i_1}(-1)E_{i_2}(-1)M^{(3)}$,  hence $\det{(M_1^{(3)})}=\det{(M^{(3)})}=1$. This proves the result for the matrix $M_1$. The proofs for the other matrices listed at {\it (ii)} and {\it (iii)} above are similar and are accordingly omitted.

Finally, regarding the last matrix $M([i,j],\alpha)$, it is readily seen  that $\det{(M([i,j],\alpha))}$ $=\det{M}=1$. Further, on  multiplying the entries of the $i$th row of the matrix $M^{(3)}$ by $\alpha^3$ and  then multiplying the  entries of the $j$th column by $\alpha^{-3}$, we get the matrix ${(M([i,j],\alpha))^{(3)}}$.  It follows that $\det{(M([i,j],\alpha)^{(3)})}=\det{(M^{(3))}}=1$. 
\end{proof}
 
\subsection{Third order unimodular matrices}\label{33matrices}
 
We will now obtain third order square integer  matrices $A=(a_{ij})$, with no $a_{ij} = \pm 1$,  such that both   $\det{(a_{ij})}$ and   $\det{(a_{ij}^3)}$ are equal to 1. We have to solve two simultaneous equations in nine independent variables.  A fair amount of computer search yielded essentially only one such matrix, namely, 
\begin{equation}
A_1=\begin{bmatrix}
7 &  11 &  2\\ 
13 &  20 &  3 \\
2 &  3 &  0
\end{bmatrix},
\label{defmatrixA1}
\end{equation}
which  satisfies the conditions  $\det{A_1}=1$ and   $\det{(A_1^{(3)})}=1$.

We now give a theorem that gives a parametric solution  of the problem.

\begin{theorem}\label{unimodular} The matrix $A$ defined by
\begin{equation}
A=\begin{bmatrix}
(16t+1)(2592t^2+288t+7) & (18t+1)(24t+1)(144t+11) & 2\\
(12t+1)(5184t^2+540t+13) & (72t+5)(1296t^2+153t+4) & 3\\
 2 & 3 & 0
\end{bmatrix},
\label{defmatrixA}
\end{equation}
where $t$ is an arbitrary parameter, satisfies the conditions   $\det{A}=1$ and   $\det{(A^{(3)})}=1$.
\end{theorem}
\begin{proof} We begin with the $3 \times 3$ matrix $B=(b_{ij})$  where we take 
\begin{equation}
b_{13}=b_{23}=b_{31}=b_{32}=1,\quad  b_{33}=0, \label{valbset1}
\end{equation}
 so that the matrix $B$ may be written as  follows:
\begin{equation}
B=\begin{bmatrix}
b_{11} & b_{12} & 1\\
b_{21} & b_{22} & 1 \\
1 & 1 & 0 
\end{bmatrix}.
\end{equation}

We then get,
\begin{equation}
\begin{aligned}
\det{B} & =-b_{11}  + b_{12} + b_{21} - b_{22}, \\
\det{(B^{(3)})} & =-b_{11}^3  + b_{12}^3 + b_{21}^3 - b_{22}^3.
\end{aligned}
\label{detB}
\end{equation}

We note that a parametric solution of the simultaneous diophantine equations,
\begin{equation}
\begin{aligned}
x_1+x_2+x_3+x_4+x_5 &=0, \\
x_1^3+x_2^3+x_3^3+x_4^3+x_5^3 &=0,
\end{aligned}
\label{eqx}
\end{equation}
given by Choudhry \cite[p.\ 316]{Ch}, is as follows:
\begin{equation}
\begin{aligned}
x_1  &  =  pq(r^2-s^2)+q^2r^2, \\
x_2  &  =  -(p^2s(r+s)-q^2rs),\\
x_3  &  =  p^2r(r+s)+pqr^2-q^2rs, \\
x_4  &  =  -(p^2r(r+s)+pq(r^2-s^2)),\\
x_5  &  =  p^2s(r+s)-pqr^2-q^2r^2,
\end{aligned}
\label{valx}
\end{equation}
where $p, q, r, $ and $s$ are arbitrary parameters.

With the values of $x_i, i=1,\ldots,5$, defined by \eqref{valx}, we take 
\begin{equation}
b_{11}=x_2, \quad b_{12}=-x_3,\quad  b_{21}=-x_4, \quad b_{22}=x_5, \label{valbset2}
\end{equation}
 when we get,
\begin{equation}
\begin{aligned}
\det{B} & =x_1 \\
\det{(B^{(3)})} & =x_1^3.
\end{aligned}
\label{detBspl}
\end{equation}

We now choose the parameters $p, q, r, s$, as follows:
\begin{equation}
p=36t+3,\quad q=-1, \quad r=144t+11, \quad s=-144t-9,
\end{equation}
where $t$ is an arbitrary parameter, when we  get $x_1=1$. The entries of the matrix $B$ may now be written in terms of the parameter $t$.  We rename this matrix as $C$, and write it explicitly as follows:
\begin{equation*}
C=\begin{bmatrix}
9(16t + 1)(2592t^2 + 288t + 7) & 6(18t + 1)(24t + 1)(144t + 11) & 1\\
6(12t + 1)(5184t^2 + 540t + 13) & 4(72t + 5)(1296t^2 + 153t + 4) & 1 \\
1 & 1 & 0 
\end{bmatrix}.
\label{matrixB0}
\end{equation*}
\normalsize

Since $x_1=1$, it follows from \eqref{detBspl} that the matrix $C$ satisfies the conditions   $\det{C}=1$ and   $\det{(C^{(3)})}=1$.

Now on starting with  the matrix  $C$, and using the last matrix listed in Lemma \ref{genlemma} four times, in succession, we  obtain four matrices $C_i, i=1, \ldots, 4$, as follows:
\begin{equation}
\begin{aligned}
C_1& =C([1, 3], 1/3), \quad  & C_2 &=C_1([2, 3], 1/2), \\
  C_3 & =C_2([3, 1], 3),\quad   & C_4& =C([3, 2], 2).
	\end{aligned}
\end{equation}

In view of Lemma \ref{genlemma}, each of the matrices  $C_i, i=1, \ldots, 4$, satisfies the conditions   $\det{C_i}=1$ and   $\det{(C_i^{(3)})}=1$. In fact, the matrix  $C_4$ is the matrix $A$ mentioned in the theorem. It follows that   $\det{A}=1$ and   $\det{(A^{(3)})}=1$.
\end{proof}

When $t=0$, the matrix $A$, defined by \eqref{defmatrixA}, reduces to the matrix $A_1$ given by \eqref{defmatrixA1}. As a second numerical example,  when  $t=1$, we get the matrix
\begin{equation*}
A_2=\begin{bmatrix}
49079 & 73625 &  2\\ 
74581 &  111881 &  3 \\
2 &  3 &  0
\end{bmatrix},
\end{equation*}
which satisfies the conditions   $\det{A_2}=1$ and   $\det{(A_2^{(3)})}=1$.

We note that one of the entries of the matrix $A$ given by Theorem~\ref{unimodular} is always zero. While it would be interesting to find a $3 \times 3$ integer matrix $A$, none of whose entries is  0 or $\pm 1$,  such that both $\det{A}$ and   $\det{(A^{(3)})}$ are equal to 1, we could not find such an example. 

\section{A more general problem}\label{detAk}
We will now find third order square integer matrices $A$ such that   $\det{A}=k $ and   $\det{(A^{(3)})}=k^3$, where $k \neq 1$ is a nonzero integer.

In fact, in Section~\ref{33matrices}, we  have already obtained  a solution to this problem in terms of four arbitrary parameters $p, q, r, s$, with $k= pq(r^2-s^2)+q^2r^2$, since  the matrix $B$, whose entries are defined by \eqref{valbset1} and \eqref{valbset2}, satisfies the conditions \eqref{detBspl} where the value of $x_1$ is given by \eqref{valx}. We note, however, that one entry of the matrix $B$ is always 0.

A computer search for  $ 3 \times 3$ integer matrices, none of the entries  being $0$ or $\pm 1$, such that   $\det{A}=k $ and   $\det{(A^{(3)})}=k^3$, where $k$ is an integer $ < 10$,  yielded just one such example, namely the matrix, 
\begin{equation}
M=\begin{bmatrix}
-5 &  4 &  10 \\
 5 &  3 &  11\\
 3 &  2 &  7
\end{bmatrix},
\end{equation}
such that $\det{M}=7$ and $\det{(M^{(3)})}=7^3$. The following theorem gives a more general solution of the problem with the   entries of the matrix $A$ being  given in terms of polynomials in six arbitrary integer parameters.

\begin{theorem}\label{Thgenk} If the polynomial $\phi(\alpha_1, \alpha_2, \alpha_3, \beta_1, \beta_2, \beta_3)$ is  defined by
\begin{multline}
\phi(\alpha_1, \alpha_2, \alpha_3, \beta_1, \beta_2, \beta_3) =-(\alpha_2\beta_3 + \alpha_3\beta_2)\alpha_1^8\alpha_2^2\alpha_3^2\beta_2^4\beta_3^4 - (\alpha_2^4\beta_3^4 - \alpha_2^3\alpha_3\beta_2\beta_3^3 \\
+ \alpha_2^2\alpha_3^2\beta_2^2\beta_3^2 - \alpha_2\alpha_3^3\beta_2^3\beta_3 + \alpha_3^4\beta_2^4)(\alpha_2\beta_3 + \alpha_3\beta_2)^2\alpha_1^7\beta_1\beta_2\beta_3- (\alpha_2\beta_3 + \alpha_3\beta_2)\\
\times (\alpha_2^4 \beta_3^4 + \alpha_2^2 \alpha_3^2 \beta_2^2 \beta_3^2 + \alpha_3^4 \beta_2^4)\alpha_1^6\alpha_2\alpha_3\beta_1^2\beta_2\beta_3 - 2\alpha_1^5\alpha_2^4\alpha_3^4\beta_1^3\beta_2^3\beta_3^3\\
 - (\alpha_2\beta_3 + \alpha_3\beta_2)(\alpha_2^4\beta_3^4 - 2\alpha_2^3\alpha_3\beta_2\beta_3^3 + \alpha_2^2\alpha_3^2\beta_2^2\beta_3^2 - 2\alpha_2\alpha_3^3\beta_2^3\beta_3 + \alpha_3^4\beta_2^4)\alpha_1^4\alpha_2^2\alpha_3^2\beta_1^4\\
 + 2(\alpha_2^2\beta_3^2 + \alpha_2\alpha_3\beta_2\beta_3 + \alpha_3^2\beta_2^2)\alpha_1^3\alpha_2^4\alpha_3^4\beta_1^5\beta_2\beta_3 + (\alpha_2\beta_3 + \alpha_3\beta_2)(\alpha_2^2\beta_3^2 + \alpha_3^2\beta_2^2)\\
\times \alpha_1^2\alpha_2^4\alpha_3^4\beta_1^6+ (\alpha_2^2\beta_3^2 + \alpha_3^2\beta_2^2)\alpha_1\alpha_2^5\alpha_3^5\beta_1^7, \label{defphi}
\end{multline}
with $\alpha_1, \alpha_2, \alpha_3, \beta_1, \beta_2, \beta_3$, being arbitrary integer parameters, the matrix $A$ defined by
\begin{equation}
A=\begin{bmatrix}
\phi(p, q, r, u, v, w) & \phi(q, r, p,  v, w, u) & \phi(r, p, q, w, u, v) \\
p  & q & r \\
u & v  & w
\end{bmatrix},
\label{defAk}
\end{equation}
satisfies the conditions,
\begin{equation}
\det{A}=k, \quad  \mbox{\rm and} \quad   \det{(A^{(3)})}=k^3, \label{valdetA}
\end{equation}
 where 
\begin{multline}
k =   pqr(pv - qu)(pw - ru)(qw - rv)(p^2v^2 + pquv + q^2u^2)\\
 \quad \times (p^2w^2 + pruw + r^2u^2)(q^2w^2 + qrvw + r^2v^2)(pqw + prv + qru). \label{valk}
\end{multline}
\end{theorem}
\begin{proof}
We begin with the matrix $A$ defined by
\begin{equation}
A=\begin{bmatrix}
x & y &z \\
p & q & r\\
u & v & w
\end{bmatrix},
\label{Axyz}
\end{equation}
when Eqs. \eqref{valdetA} may be written as follows:
\begin{align}
(qw - rv)x + (ru-pw )y + (pv - qu)z  & = k, \label{condk}\\
(q^3w^3 - r^3v^3)x^3 + (r^3u^3-p^3w^3)y^3 + (p^3v^3 - q^3u^3)z^3  & = k^3. \label{condkcub}
\end{align}

On eliminating $k$ from Eqs. \eqref{condk} and \eqref{condkcub}, we get,
\begin{multline}
(q^3w^3 - r^3v^3)x^3 + (r^3u^3-p^3w^3)y^3 + (p^3v^3 - q^3u^3)z^3 \\
-((qw - rv)x + (ru-pw )y + (pv - qu)z)^3=0.
\label{eccubic}
\end{multline}
We note that when $(x, y, z)=(p, q, r)$, both $\det{A}$ and $\det{(A^{(3)})}$ vanish, and hence $(x, y, z)=(p, q, r)$ is a solution of Eq. \eqref{eccubic}. Similarly, $(x, y, z)=(u, v, w)$ is also a solution of Eq. \eqref{eccubic}.

Equation \eqref{eccubic} is  a homogeneous cubic equation in the variables $x, y$ and $z$, and accordingly, we may consider it as an elliptic curve in the projective plane $\mathbb{P}^2$ with two known points on the curve being $P_1=(p, q, r)$ and $P_2=(u, v, w)$. 
If we draw a line joining the points $P_1$ and $P_2$ to intersect the elliptic curve \eqref{eccubic} in a third rational point, say $(x_1, y_1, z_1)$,  and take $(x, y, z)=(x_1, y_1, z_1)$, the left-hand side of both Eqs. \eqref{condk} and \eqref{condkcub} becomes $0$, and we do not get a nonzero value of $k$ as desired. Accordingly, we draw a tangent at the point $P_1$ to intersect the elliptic curve in a point $P_3$  whose coordinates are as follows:
\begin{equation}
x=\phi(p, q, r, u, v, w),\;  y=\phi(q, r, p,  v, w, u),\;  z=\phi(r, p, q, w, u, v), \label{valxyz}
\end{equation}
where $\phi(\alpha_1, \alpha_2, \alpha_3, \beta_1, \beta_2, \beta_3) $ is defined by \eqref{defphi}.

The values of $x, y, z$, given by \eqref{valxyz} satisfy Eq. \eqref{eccubic}, and further, the value of $k$ obtained from Eq. \eqref{condk} is given by \eqref{valk}. On substituting the values of $x, y, z$, given by \eqref{valxyz} in \eqref{Axyz}, we get the matrix $A$,  defined by \eqref{defAk}, that satisfies the conditions \eqref{valdetA} with the value of $k$ being given by \eqref{valk} as stated in the theorem. 
\end{proof}

We note that in any numerical example of the matrix $A$, if the greatest common divisor $g$ of the entries of any row (or column) of the matrix $A$ is $> 1$, then $g$ is also a factor of $\det{A}$, and we can divide the entries of that row (or column) by $g$ to get a numerically smaller example of a matrix satisfying the specified conditions. For instance, when $(p, q, r, u, v, w)=(2, -3, 3, 3, -2, 4)$, we get a matrix which, on factoring out the greatest common divisor of the entries of the first row, yields the matrix,
\begin{equation*}
A = \begin{bmatrix}
-57797 & -109147 & -22789\\
 2 & -3 & 3\\
 3 & -2 & 4
\end{bmatrix},
\end{equation*}
for which $\det{A}=123690$ and $\det{(A^{(3)})}=123690^3$.

\bigskip

\noindent Postal Address: Ajai Choudhry, 13/4 A Clay Square,\\
\hspace*{1.1 in} Lucknow - 226001, India.

\noindent E-mail address: ajaic203@yahoo.com

\end{document}